\documentclass{amsart}

\usepackage{hyperref,url,arydshln} 
\usepackage[all]{xy}
\usepackage[margin=1.6in]{geometry}
\newtheorem{theorem}{Theorem}[section]
\newtheorem*{theorem*}{Theorem}
\newtheorem{corollary}[theorem]{Corollary}
\newtheorem{lemma}[theorem]{Lemma}
\newtheorem{proposition}[theorem]{Proposition}
\theoremstyle{definition}
\newtheorem{example}[theorem]{Example}

\theoremstyle{remark}
\newtheorem{remark}[theorem]{Remark}

\def\GL{{\mathrm{GL}}}
\def\FF{{\mathbb{F}}}
\def\CC{{\mathbb{C}}}
\def\ZZ{{\mathbb{Z}}}
\def\SS{{\mathfrak S}}
\def\tt{{\mathbf{t}}}

\def\Fl{{\mathcal{F}\ell}}
\def\wt{{\mathrm{wt}}}
\newcommand{\qbinom}[3]{\genfrac{[}{]}{0pt}{}{#1}{#2}_{#3}}

\author{Andrew Berget}
\address{Department of Mathematics, University of California, Davis, CA 95616}
\email{berget@math.ucdavis.edu}
\author{Jia Huang}
\address{Department of Mathematics, University of Minnesota, Minneapolis, MN 55455}
\email{huang338@math.umn.edu}

\title{Cyclic Sieving of
  Finite Grassmannians and Flag Varieties}

\thanks{Berget was partially supported by NSF grant DMS-0636297. Huang
  was partially supported by NSF grant DMS-1001933.}

\begin{document}

\begin{abstract}
  In this paper we prove instances of the cyclic sieving phenomenon
  for finite Grassmannians and partial flag varieties, which carry
  the action of various tori in the finite general linear group
  $\GL_n(\FF_q)$. The polynomials involved are sums of certain
  weights of the minimal length parabolic coset representatives
  of the symmetric group $\mathfrak{S}_n$, where the weight of a
  coset representative can be written as a product over its inversions.
\end{abstract}

\pagestyle{plain}

\maketitle

\section{Introduction}

The {\it cyclic sieving phenomenon (CSP)} was introduced by Reiner,
Stanton and White \cite{RSW}, generalizing Stembridge's \emph{$q=-1$
  phenomenon} \cite{stembridge}.  It has been the subject of a flurry
of recent papers in combinatorics, including a survey by Sagan
\cite{sagan}. As defined in \cite{BER}, the CSP pertains to a finite
set $X$, carrying a permutation action of a finite abelian group
written explicitly as a product $C:=C_1 \times \cdots \times C_m$ of
cyclic groups $C_i$, and a polynomial $X(\tt):=X(t_1,\ldots,t_m)$ in
$\ZZ[\tt]$. The polynomial is thought of as a generating function
for the elements of $X$ according to some natural statistic.  One
says that the triple $(X,X(\tt),C)$ {\it exhibits the CSP} if after
choosing embeddings of groups $\omega_i:C_i \hookrightarrow
\CC^\times$, one has for every $c=(c_1,\ldots,c_m)$ in $C$ that the
cardinality of the fixed point set of $c$ is given by
\[
\left[ X(\tt) \right]_{t_i=\omega_i(c_i)} = |\{x \in X: c(x)=x\}|.
\]

The goal of this paper is to prove several instances of the CSP that
pertain to the action of a maximal torus on a finite Grassmannian,
and more generally on a finite partial flag variety. We will always
take our flag varieties to be defined over $\FF_q$, a finite field
with $q$ elements. Our main theorem is technical to state
explicitly, so here we content ourselves with the following
rendition:
\begin{theorem*}
  Let $n$ and $k$ be positive integers, $k \leq n$ and
  $\alpha=(\alpha_1,\ldots,\alpha_\ell)$ a composition of $n$.
  There is an action of the torus $T_\alpha :=
  \prod_{i=1}^\ell \FF_{q^{\alpha_i}}^\times$ on the Grassmannian
  $G_k(n)$ of $k$-dimensional $\FF_q$-subspaces of $\FF_q^n$.

  For each partition $\lambda$ in a $k$-by-$(n-k)$ box there is an
  associated weight $\wt(\lambda;\alpha,k)$,
  which is a product over the cells of $\lambda$ and is a
  polynomial in $t_1,\dots,t_\ell$, such that the triple
  \[
  \left( G_k(n),\ \sum_{\lambda \subset (n-k)^k}
  \wt(\lambda;\alpha,k),\ T_\alpha \right)
  \]
  exhibits the cyclic sieving phenomenon.

More generally, if $\beta=(\beta_1,\ldots,\beta_k)$ is a composition
of $n$, and $\Fl(\beta)$ is the variety of partial flags of
subspaces of $\FF_q^n$ whose dimensions are given by the partial
sums $\beta_1+\cdots+\beta_i$, then there is an action of $T_\alpha$
on $\Fl(\beta)$ and a polynomial $X_{\alpha,\beta}(\tt)$ such that
$(\Fl(\beta), X_{\alpha,\beta}(\tt) ,T_\alpha)$ exhibits the cyclic
sieving phenomenon.
\end{theorem*}

Some remarks are in order. The first is that the determination of
the weights $\wt(\lambda;\alpha,k)$ is entirely elementary. It
consists of breaking up $\lambda$ into smaller partitions,
determined by $\alpha$ and $k$, and then evaluating the relative
positions of the cells of the smaller partitions within certain
boxes. For partial flag varieties the polynomials
$X_{\alpha,\beta}(\tt)$ are determined in a similar fashion.

The second remark is that in the extreme case $\alpha= (n)$, our
polynomials $X_{\alpha,\beta}(\tt)$ coincide with the
$(q,t)$-multinomial coefficients $\qbinom{n}{\beta}{q,t}$ of Reiner,
Stanton and White \cite[\S9]{RSW}. In this special case, our main
theorem was one of the foundational examples of the CSP in
\cite{RSW}.

Thirdly, the weights will visibly have the property
\[
\wt(\lambda;\alpha,k)|_{t_i= 1} = q^{|\lambda|}
\]
where $|\lambda|$ is the sum of its parts. Thus, our formula is a
carefully crafted refinement of the well-known expression
\[
|G_k(n)| = \sum_{\lambda \subset (n-k)^k} q^{|\lambda|}.
\]
It is equally well-known that this expression factors as
\[
|G_k(n)| = \sum_{\lambda \subset (n-k)^k} q^{|\lambda|} = \frac{(q^n
  -1)(q^n - q)\dots (q^n-q^{n-k+1})}{(q^k-1)(q^k-q)\dots (q^k -
  q^{k-1})}.
\]
For general $\alpha$ the generating function $\sum_\lambda
\wt(\alpha;\lambda,k)$ does not appear to have a similar
factorization. However, we will show that a natural refinement of it
does, and specializing at $t_1,\ldots,t_\ell=1$ gives rise to the
well-known $q$-Vandermonde identity. In fact, our proof for the CSP
on partial flag varieties gives a generalized $q$-Vandermonde
identity
\[
  \qbinom{\alpha_1+\cdots+\alpha_\ell}{\beta_1,\cdots,\beta_m}{q}
  =\sum_{\beta^{(1)},\ldots,\beta^{(\ell)}}
  \prod_{r=1}^\ell\qbinom{\alpha_r}{\beta^{(r)}}{q}
  \prod_{1\leq s\leq r\leq\ell} \prod_{1\leq i<j\leq
  m}q^{\beta^{(r)}_i\beta^{(s)}_j},
  \]
summed over all weak compositions
$\beta^{(r)}=(\beta^{(r)}_1,\ldots,\beta^{(r)}_m)$ of $\alpha_r$,
$1\leq r\leq\ell$, with component-wise sum
$\beta^{(1)}+\cdots\beta^{(\ell)}=\beta$.

Lastly, there is also a ``$q=1$'' version of our main result on
flags that pertains to flags of sets. For $r=1,\ldots,\ell$, let
$C_r$ be a cyclic group generated by a \emph{regular element} $c_r$
in the symmetric group  $\mathfrak S_{\alpha_r}$, {\it i.e.}, $C_r$
acts \emph{nearly freely} on $[\alpha_r]$. The triple
\[
\left({[\alpha_r]\choose \beta^{(r)}},\
\qbinom{\alpha_r}{\beta^{(r)}}{t_r},\ C_r\right)
\]
exhibits the CSP, where $\beta^{(r)}$ is a weak composition of
$\alpha_r$, and  ${[\alpha_r]\choose \beta^{(r)}}$ is the set of all
partial flags of subsets of cardinality
$\beta^{(r)}_1,\beta^{(r)}_1+\beta^{(r)}_2,\ldots$ in the set
$[\alpha_r]=\{1,\ldots,\alpha_r\}$; see \cite[\S1]{RSW}. The
embedding $\mathfrak S_{\alpha_1}\times\cdots\times\mathfrak
S_{\alpha_\ell} \hookrightarrow \mathfrak{S}_n$ sends
$C_1\times\cdots\times C_\ell$ to a subgroup $T_\alpha$ of
$\mathfrak S_n$. A moment's thought reveals that the triple
\[
\left({[n]\choose\beta},\ Y_{\alpha,\beta}(\tt),\ T_\alpha\right)
\]
exhibits the cyclic sieving phenomenon for any composition $\beta$
of $n$, with
\begin{equation}\label{eq:q=1}
Y_{\alpha,\beta}(\tt) = \sum_{\beta^{(1)},\ldots ,\beta^{(\ell)}}
\qbinom{\alpha_1}{\beta^{(1)}}{t_1}\cdots
\qbinom{\alpha_\ell}{\beta^{(\ell)}}{t_\ell},
\end{equation}
summer over all weak compositions $\beta^{(r)}$ of $\alpha_r$,
$r=1,\ldots,\ell$, whose component-wise sum is $\beta$. One will see
in Remark~\ref{rem:CSPflag} that $Y_{\alpha,\beta}(\tt)$ can be
viewed as a ``$q=1$'' version of $X_{\alpha,\beta}(\tt)$.

Our paper is organized as follows. In Section~\ref{sec:precise} we
set up the complete statement of our main theorem for cyclic sieving
of finite Grassmannians. In Section~\ref{sec:proof} we then prove
the main theorem, which will follow from a series of lemmas. In
Section~\ref{sec:flags} we extend the main result to the finite flag
varieties.

\section{Precise statement of the main result}\label{sec:precise}

In this section we carefully define the polynomials involved in the
statement of our main theorem, as well as the torus
actions. Throughout this discussion we have fixed two integers $n$ and
$k$, $k \leq n$, which define the Grassmannian that we are working in.

\subsection{Partition statistics}\label{sec:partition}
A (weak) composition $\alpha=(\alpha_1,\dots,\alpha_\ell)$ of $n$ is
a sequence of positive (nonnegative) integers $\alpha_i$ such that
$n=\alpha_1+\cdots+\alpha_\ell$. The integers $\alpha_i$ are called
\emph{parts of $\alpha$}. Fix a composition $\alpha$ with exactly
$\ell$ parts.

Associated to each partition $\lambda\subset (n-k)^k$ is a
\emph{(reduced) row echelon form}, which is a matrix of $1$'s
(\emph{pivots}), $0$'s, and $*$'s. The matrix is in row echelon
form, in the sense that every pivot $1$ sees zeros in directions
east, north and south. The columns containing the $1$'s are called
\emph{pivot columns}, and their positions are $\{\lambda_i+k-i+1: 1
\leq i \leq k\}$. The remaining spots are filled with stars so that
the $i$-th row contains exactly $\lambda_i$ stars.

The composition $\alpha$ breaks up the columns of this matrix into
blocks of sizes $\alpha_1$,\dots, $\alpha_\ell$, labeled by
$1,\ldots,\ell$, \emph{from right to left}. We set $\beta(\lambda) =
(\beta_1,\dots,\beta_\ell)$ so that $\beta_j$ is the number of $1$'s
in the $j$-th block of columns. To be formal,
\[
\beta_j:= \#\{i: \alpha_\ell + \dots + \alpha_{j+1} <
\lambda_i+k-i+1 \leq \alpha_\ell + \dots + \alpha_j\}.
\]
Note that $\beta(\lambda)$ is a weak composition of $k$ into $\ell$
parts, and this association is not injective.

\begin{example}
  Take $n=9$, $k=4$ and $\alpha = (4,2,3)$. The composition
  associated to $\lambda = (5,4,1,1)$
  is $\beta(\lambda) = (2,0,2)$, as is evidenced by the matrix below
  or using the formula above.
  \[
   \left[\begin{tabular}{c|c|c}
  $*\ 0\ 0$ & $*\ *$ & $*\,0\,*\,1$   \\
  $*\ 0\ 0$ & $*\ *$ & $*\,1\ 0\ 0$   \\
  \hline\hline
  $*\ 0\ 1$                 \\
  $*\ 1\ 0$                 \\
  \end{tabular}\right]
  \]
  We have indicated the composition $\alpha$ as blocks of columns, and
  the composition $\beta$ as blocks of rows (the double line has an
  empty block in between its lines).
\end{example}

As above, we can partition the row echelon form of every $\lambda$
into a block triangular matrix. The $(r,s)$-th block is of size
$\beta_r$-by-$\alpha_s$, and the $*$'s in it give rise to a
partition that fits in a $\beta_r$-by-$(\alpha_s-\beta_s)$ box.
Denote this partition by $\lambda^{r,s}$, which only depends on
$\alpha$, $\lambda$ and $k$. Note that $\lambda^{r,r}$ can be
arbitrary, while $\lambda^{r,s}$ is always a
$\beta_r$-by-$(\alpha_s-\beta_s)$ rectangle if $r<s$, or empty if
$r<s$. By sending $\lambda$ to $[\lambda^{r,s}]$ one obtains
bijection
\begin{multline}\label{eq:lambda}
\left\{\lambda:\beta(\lambda)=\beta\right\}\xrightarrow{\sim} \\
\left\{[\lambda^{r,s}]:\lambda^{r,r}\subseteq(\alpha_r-\beta_r)^{\beta_r},
\lambda^{r,s}=(\alpha_s-\beta_s)^{\beta_r}\ (r<s),\
\lambda^{r,s}=\emptyset\ (r>s)\right\}.
\end{multline}

Denote by $[a,b]$ the \emph{$q$-number}
\[
[a,b]:= \frac{a^q - b^q}{a-b} = \sum_{i+j=q-1} a^i b^j.
\]
If $a$ and $b$ are $(q-1)$-th roots of unity then $[a,b]=1$ when $a
\neq b$ or $[a,b]=q$ when $a=b$.

We define the \emph{weight of a cell $x \in \lambda^{r,s}$} to be
\[
\wt(x;t_r,t_s) :=
 \left[ t_r^{q^{i(x)+j(x)}}, t_s^{q^{i(x)+\beta_r}} \right].
\]
Here $i(x)$ and $j(x)$ are the horizontal and vertical distances of
$x$ from the bottom-left cell of the $\beta_r$-by-$(\alpha_s
-\beta_s)$ box in which $\lambda^{r,s}$ fits. Define the
\emph{weight of $\lambda^{r,s}$} to be
\[
\wt(\lambda^{r,s};t_r,t_s):=\prod_{x\in\lambda^{r,s}}
\wt(x;t_r,t_s).
\]
Define the \emph{weight of $\lambda$} to be
\[
\wt(\lambda;\alpha,k):=\prod_{1 \leq r \leq s \leq \ell}
\wt(\lambda^{r,s};t_r,t_s).
\]
We see at once that $\wt(\lambda;\alpha,k)$ is a polynomial in
$t_1,\dots,t_\ell$. This choice of a weight was directly influenced
by the results of Reiner and Stanton in \cite{ReinerStanton}. In the
case $\alpha=(n)$ our weight is exactly the weight they associate to
$\lambda$.

\begin{example}
  As before, take $n=9$, $k=4$, $\alpha =(4,2,3)$, and
  $\lambda = (5,4,1,1)$. The labeling of all the cells of $\lambda$
  with their weights gives the following diagram:
  \[
   \left[\begin{tabular}{c|cc|cc}
  $[t_1^q,t_3^{q^2}]$ & $[t_1^q,t_2^{q^2}]$ & $[t_1^{q^2},t_2^{q^3}]$ & $[t_1^q,t_1^{q^2}]$ & $[t_1^{q^2},t_1^{q^3}]$ \\
  $[t_1,t_3^{q^2}]$   & $[t_1,t_2^{q^2}]$   & $[t_1^q,t_2^{q^3}]$     & $[t_1,t_1^{q^2}]$   \\
  \hline\hline
  $[t_3^q,t_3^{q^2}]$                 \\
  $[t_3,t_3^{q^2}]$                 \\
  \end{tabular}\right]
  \]
\end{example}

\begin{example}
  If $\alpha=(1,1,\dots,1)$ then $\beta(\lambda)$ is what is sometimes
  referred to as the ``code'' associated to $\lambda$. That is,
  $\beta(\lambda)$ has a $1$ in positions $n-k-\lambda_i + i$, $1 \leq
  i \leq k$, and zeros elsewhere. The weight of $\lambda$ is
  \[
  \wt(\lambda;1^n,k) = \prod_{\substack{1 \leq i \leq k \\1 \leq j \leq \lambda_j}}
  [ t_{k-\lambda_j'+j} ,  t_{n-k-\lambda_i + i}].
  \]
  Here, $\lambda'$ denotes the conjugate partition, and $\lambda_j'$
  is its $j$-th part, {\it i.e.} the number of parts of $\lambda$
  that are at least $j$.
\end{example}

\subsection{Tori in $\GL_n(\FF_q)$}
Let $\GL_n(K)$ denote the group of $n$-by-$n$ matrices with non-zero
determinant in the field $K$. If $K$ is an algebraically closed
field then a maximal torus in $\GL_n(k)$ is a subgroup conjugate to
the group of diagonal matrices $(K^\times)^n \subset \GL_n(K)$.
Since $\FF_q$ is not algebraically closed, the definition of a torus
in this group is more subtle.

Let $F$ be denote Frobenius automorphism $F(x)= x^q$ of the
algebraic closure $K$ of $\FF_q$. A maximal torus $T \subset
\GL_n(K)$ is $F$~-~stable if $F(T) = T$. A \emph{maximal torus} in
$\GL_n(\FF_q)$ consists of the $F$~-~fixed points of an $F$~-~stable
torus $T$. See Carter \cite[Chapter 3]{carter} for more details
related to these objects.

We now define a class of maximal tori in $\GL_n(\FF_q)$ labeled by
compositions of $n$. Given a composition $\alpha =
(\alpha_1,\dots,\alpha_\ell)$ of $n$, consider the $n$-dimensional
$\FF_q$-vector space
\[
V_\alpha:=\FF_{q^{\alpha_1}} \oplus \dots \oplus
\FF_{q^{\alpha_\ell}}.
\]
The multiplications of the factors give rise to $\FF_q$-linear
automorphisms of $V_\alpha$. Choosing an isomorphism $V_\alpha \cong
\FF_q^n$ one has an injection of groups
\[
T_\alpha:=\FF_{q^{\alpha_1}}^\times \times \dots \times
\FF_{q^{\alpha_\ell}}^\times \hookrightarrow \GL_n(\FF_q)
\]
whose image is a maximal torus in $\GL_n(\FF_q)$ (this will be seen
explicitly in Corollary~\ref{cor:action}).

The Grassmannian $G_k(V_\alpha)\cong G_k(n)$ of the $k$-dimensional
subspace of $\FF_q^n$ can be thought of as the collection of full
rank $k$-by-$n$ matrices with entries in $\FF_q$, modulo the natural
left action by $\GL_k(\FF_q)$. The group $\GL_n(\FF_q)$ acts on the
right of $G_k(n)$ by multiplication. It follows that after embedding
$T_\alpha$ in $\GL_n(\FF_q)$ there is a right action of $T_\alpha$
on $G_k(n)$.
\begin{proposition}
  The number of fixed points of an element $T_\alpha$ as it acts on
  the Grassmannian $G_k(n)$ does not depend on the embedding of
  $T_\alpha$ in $\GL_n(\FF_q)$.
\end{proposition}
\begin{proof}
  Suppose we have two different injections $\phi,\psi:T_\alpha
  \hookrightarrow \GL_n(\FF_q)$, both coming from isomorphisms
  $V_\alpha\cong\FF_q^n$, as above. We see that $\phi(t)$ and
  $\psi(t)$ are conjugate in $\GL_n(\FF_q)$ for any $t\in T_\alpha$.
  The fixed sets of $\phi(t)$ are $\psi(t)$ are translates of each
  other, hence equicardinal.
\end{proof}

In the case that $\alpha=(1^n)$, the torus $T_\alpha$ is said to be
\emph{maximally split}. When we are in the case of a maximally split
torus acting on a Grassmannian or a partial flag variety our
analysis will be considerably simplified.

\subsection{Precise statement of the theorem}
We are now in a position to give the full statement of our main
theorem, where the notation is as in the previous subsections.

\begin{theorem}[\textit{c.f.} Reiner--Stanton \cite{ReinerStanton}]\label{thm:main}
  The triple
  \[
  \left( G_k(V_\alpha), \sum_{\lambda \subset (n-k)^k}
    \wt(\lambda;\alpha,k) ,\ T_\alpha \right)
  \]
  exhibits the cyclic sieving phenomenon.
\end{theorem}

\begin{corollary}
  If $(t_1,\dots,t_n)$ is an element of a maximally split torus of
  $\GL_n(\FF_q)$ then it fixes exactly
  \[
  \sum_{\lambda \subset (n-k)^k} \prod_{\substack{1 \leq i \leq k\\
      1 \leq j \leq \lambda_i}}
  \frac{\omega(t_{k-\lambda_j'+j})^q-\omega(t_{n-k-\lambda_i + i})^q}
  {\omega(t_{k-\lambda_j'+j})-\omega(t_{n-k-\lambda_i + i})}
  \]
  elements in $G_k(n)$, where $\omega$ is an injection of groups
  $\FF_q^\times \hookrightarrow \CC^\times$.
\end{corollary}
\begin{proof}
  This is the special case of the theorem when $\alpha=(1^n)$.
\end{proof}
The theorem gives an effective method for computing certain complex
characters of $\GL_n(\FF_q)$ at elements that are semi-simple over
an algebraic closure of $\FF_q$. Let $\chi$ be the character of the
induction $\operatorname{Ind}_{P_k}^{\GL_n(\FF_q)} \CC$ of the
trivial representation $\CC$ of $P_k$, the parabolic subgroup of
block upper triangular matrices with two invertible diagonal blocks
of size $k$ and $n-k$. An element of $\GL_n(\FF_q)$ is semi-simple
over the algebraic closure of $\FF_q$ if and only if it occurs in
the image of some $T_\alpha$. A moment's thought reveals that the
computation in Theorem~\ref{thm:main} is equivalent to the
computation of $\chi$ at semi-simple elements.

This character can be computed using the classical formula for
induction. However, this formula is of little use to us since it
simply states the tautology: The induced character evaluated at $u$
is the number of fixed points $u$ on $\GL_n(\FF_q)/P_k=G_k(n)$. Our
theorem reduces the computation of this character to fewer than
\[
\binom{n}{k} \cdot k(n-k)
\]
evaluations of rational functions of the form $[a,b] =
(a^q-b^q)/(a-b)$ at roots of unity. This can be done using floating
point arithmetic, since all answers can be rounded to the nearest
integer.

\section{Proof of Theorem~\ref{thm:main}}\label{sec:proof}
We now begin to prove our main theorem. To do this we give a number
of preliminary lemmas. Once we have these in hand we will explicitly
count the number of subspaces fixed by a given torus element, using
the Cecioni--Frobenius theorem and a result of Reiner and Stanton.

As before, $\alpha$ is a composition of $n$ and
\begin{align*}
  V_\alpha &= \FF_{q^{\alpha_1}} \oplus \dots \oplus \FF_{q^{\alpha_\ell}},\\
  T_\alpha &=\FF_{q^{\alpha_1}}^\times \times \dots \times
  \FF_{q^{\alpha_\ell}}^\times.
\end{align*}
There is a right action of $T_\alpha$ on $G_k(n)$ that comes from
choosing an isomorphism $V_\alpha \cong \FF_q^n$, and hence an
embedding $T_\alpha \hookrightarrow \GL_n(\FF_q)$.

\begin{lemma}\label{companion}
  Let $u$ be an element in $\mathbb F_{q^n}^\times$ with minimal   polynomial $$x^d - a_{d-1} x^{d-1} - \dots - a_0.$$
  Then there exists a basis for $\FF_q^n$ such that the right action of $u$ on $G_k(n)$ is represented
  by the right multiplication by a block diagonal matrix
  $\operatorname{diag}(U,\ldots,U)$ in $\GL_n(\FF_q)$, where
  $$
  U=\begin{bmatrix}
    0 & 1 \\
    & 0 & 1 \\
    & & \cdots \\
    & & & 0 & 1 \\
    a_0 & a_1 & \cdots & & a_{d-1}
  \end{bmatrix}
  $$
  is the companion matrix of the minimal polynomial of $u$.
\end{lemma}

\begin{proof}
  Since there is a tower of fields $\FF_q \subseteq \FF_q[u] \subseteq
  \FF_{q^n}$, we see that $\FF_{q^n}$ is of degree $n/d$ over
  $\FF_q[u]$. Let $v_1,\dots,v_{n/d}$ denote an $\FF_q[u]$-basis for
  $\FF_{q^n}$. It follows from Lang \cite[Proposition V.1.2]{lang}
  that
  \[
  \{u^i v_j: 0 \leq i \leq d-1, 1 \leq j \leq n/d\}
  \]
  is a $\FF_q$-basis for $\FF_{q^n}$. There is now an isomorphism of
  $\FF_q[u]$-modules,
  \[
  \FF_{q^n} \cong (\FF_q[u])^{\oplus n/d},
  \]
  given by grouping the basis elements above by the index of $v$.  The
  matrix of $u$ acting on each copy of $\FF_q[u]$ is $U$, so we are done.
\end{proof}

\begin{corollary}\label{cor:action}
  Let $u = (u_1,\dots,u_\ell)$ be an element of $T_\alpha$.
  Then there exists a basis for $V_\alpha$ such that the action
  of $u$ on $V_\alpha$ is represented by the right
  multiplication by the block diagonal matrix
  \[
  \operatorname{diag}(U_1,\dots,U_1,U_2,\dots,U_2,\dots,U_\ell,\dots,U_\ell).
  \]
  Here, as before, $U_i$ is the companion matrix of the minimal
  polynomial of $u_i$ and each matrix $U_i$ is repeated
  $\left[\FF_{q^{\alpha_i}}:\FF_q[u_i]\right]$ times.
\end{corollary}

If $Z \subset V_\alpha$ is a $k$-dimensional subspace, we obtain a
sequence of nonnegative integers
$\beta(Z)=(\beta_1,\ldots,\beta_\ell)$ by letting
\[
\beta_r = \dim \pi_r( \ker(\pi_{r-1}) \cap Z )
\]
where $\pi_r$ is the projection
$V_\alpha=\bigoplus_{i=1}^\ell\FF_{q^{\alpha_i}}
\rightarrow \bigoplus_{i=1}^r \FF_{q^{\alpha_i}}$.

The numbers $\beta(Z)$ can be determined in coordinates quite
easily. Choose a basis for $V_\alpha$ which is a concatenation of
bases for $\FF_{q^{\alpha_\ell}},\ldots,\FF_{q^{\alpha_1}}$. Write
the subspace $Z$ as the row space of a $k$-by-$n$ matrix in row
echelon form with respect this basis, and let $\lambda$ be the
associated partition. Then $\alpha$ has partitioned the columns of
this matrix into blocks, labeled by $1,\ldots,\ell$ \emph{from right
to left}, and $\beta_r$ is the number of $1$'s in the $r$th block.
In other words, $\beta(Z)=\beta(\lambda)$ as in
Section~\ref{sec:partition}. It is a weak composition of $k$ into
$\ell$ parts, breaking the rows of a $k$-by-$n$ matrix into blocks,
labeled by $1,\ldots,\ell$ from top to bottom (some may be empty).
Therefore one can write $Z=[Z_{rs}]_{r,s=1}^\ell$ as a block matrix.

\begin{lemma}\label{lem:structure}
  Let $u=(u_1,\dots,u_\ell) \in T_\alpha$ with $[\FF_q[u_r]:\FF_q]=d_r$.
  Let $Z$ be a $k$-dimensional subspace of $V_\alpha$ with
  $\beta(Z) = (\beta_1,\dots,\beta_\ell)$. If $Z$ is fixed by $u$, then
  under the basis given in Corollary~\ref{cor:action}, $Z$ has row echelon
  form equal to a block matrix $[Z_{rs}]_{1 \leq r \leq s \leq \ell}$
  where each block $Z_{rs}$ has dimension $\beta_r$-by-$\alpha_s$ and
  $Z_{rs}=0$ whenever $r>s$. Furthermore, each anti-diagonal block
  $Z_{rr}$ is in a block row echelon form
   \[Z_{rr}=\begin{bmatrix}
 *\ldots* &    0   & \ldots &  0 & *\ldots* & 0 & *\ldots*  &I_{d_r} & 0\ldots0 \\
 *\ldots* &    0   & \ldots &   0 & *\ldots* & I_{d_r}             \\
 *\ldots* &    0   & \ldots & I_{d_r}                 \\
  \vdots  & \vdots \\
 *\ldots* & I_{d_r}
    \end{bmatrix}.
    \]
 Here the blocks $0$, $I_{d_r}$, and $*$ are all of size $d_r$-by-$d_r$.
\end{lemma}
\begin{proof}
  Let $W_r \subseteq \FF_{q^{\alpha_r}}$ be defined as
  \[
  W_r=\pi_r( \ker(\pi_{r-1}) \cap Z ) \subseteq \FF_{q^{\alpha_r}},
  \]
  where $\pi_r$ is the natural projection $V_\alpha \to
  \bigoplus_{i=1}^r \FF_{q^{\alpha_i}}$. By definition, $W_r$ is a
  $\beta_r$-dimensional subspace of $\FF_{q^{\alpha_r}}$.

  We have $W_r u_r = W_r$, since $Z$ being fixed by $u$ implies
  \[
  \pi_r( \ker(\pi_{r-1}) \cap Z )u_r = \pi_r( (\ker(\pi_{r-1}) \cap
  Z)u ) = \pi_r( (\ker(\pi_{r-1}) \cap Z) ).
  \]
  Thus $W_r$ is an $\FF_q[u_r]$-submodule of $\FF_{q^{\alpha_r}}$.
  Using a basis of $\FF_{q^{\alpha_r}}$ furnished by the isomorphism
  (see the proof of Lemma~\ref{companion})
  \[
  \FF_{q^{\alpha_r}} \cong (\FF_q[u_r])^{\oplus \alpha_r/d_r},
  \]
  we write $W_r$ as a block matrix in row echelon form.
  Each pivot block must be an identity matrix, since
  $\FF_q[u_r]$ is irreducible as a module over itself.

  Doing the above procedure for each $r$ proves that $Z$ can be
  written as a block triangular matrix $[Z_{rs}]$ where the
  anti-diagonal blocks have the desired form. Then one can do
  block row reduction using anti-diagonal blocks to put the
  entire matrix $[Z_{rs}]$ in row echelon form.
\end{proof}

\begin{remark}
The $*$'s in $Z_{rr}$ are $d_r$-by-$d_r$ matrices with entries in
$\FF_q$ that satisfy the equation $U_rX=XU_r$, $U_r$ being the
companion matrix of the minimal polynomial of $u_r$. The solutions
to this equation bijectively correspond to elements in
$\FF_{q^{d_r}}=\FF_q[u_r]$.
\end{remark}

\begin{lemma}[Cecioni--Frobenius]
  Let $A$ and $B$ be matrices of size $a$-by-$a$ and $b$-by-$b$ with
  entries in a field $\FF$. The $\FF$-dimension of the space of solutions
  of the equation $AX=XB$, for $X$ a matrix of size $a$-by-$b$, is given by the sum
  \[
  \sum_{i,j} \deg \gcd( d_i(A),d_j(B) ),
  \]
  where $d_i(A)$ is the $i$th invariant factor of $\lambda I_a-A$
  over $\mathbb F[\lambda]$, and similarly for $d_j(B)$.
\end{lemma}
\begin{proof}
  This is a classical, if somewhat unknown result. See
  \cite[Theorem~46.3]{AX=XB}.
\end{proof}

\begin{lemma}\label{UpperBlock}
  For $i=1,2$, let $u_i$ be an element in $\mathbb F_{q^{\alpha_i}}$
  with $\mathbb F_q[u_i]=\mathbb F_{q^{d_i}}$, let $U_i$
  be the companion matrix of the minimal polynomial of $u_i$,
  and let $m_i$ be a positive integer divisible by $d_i$.
  Then the number of $m_1$-by-$m_2$ matrices $X$ with entries
  in $\FF_q$ that satisfy
  \begin{equation}\label{U1X=XU2}
  {\rm diag}(U_1,\ldots,U_1)X=X{\rm diag}(U_2,\ldots,U_2)
  \end{equation}
  is equal to
  \[
  \prod_{j=1}^{m_1}\prod_{i=1}^{m_2}
  \left[\omega_1(u_1)^{q^{i+j}},\omega_2(u_2)^{q^{i+m_1}}\right]
  \]
  where $\omega_i:\FF_{q^{\alpha_i}}^\times\hookrightarrow \CC^\times$
  is a fixed injection of groups, $i=1,2$.
\end{lemma}
We extend our previous notation and let $[a,b]_q = (a^q-b^q)/(a-b)$,
since we occasionally use the notation $[a,b]_{q^2}$, etc..
\begin{proof}
Write $X$ as a block matrix $[X_{rs}]$ where each block $X_{rs}$ has
dimension $d_1$-by-$d_2$. Then Equation (\ref{U1X=XU2}) is
equivalent to $U_1X_{rs}=X_{rs}U_2$ for all $X_{rs}$.

Take an arbitrary block $X_{rs}$ and let $z$ its the bottom-left
entry. It follows from $\omega_i(u_i)^{d_i}=\omega_i(u_i)$, $i=1,2$,
that
\begin{eqnarray*}
\prod_{x\in X_{rs}}
\left[\omega_1(u_1)^{q^{i(x)+j(x)}},\omega_2(u_2)^{q^{i(x)+m_1}}\right]
&=&\prod_{j=0}^{d_1-1} \prod_{i=0}^{d_2-1}
\left[\omega_1(u_1)^{q^{i(z)+j(z)+i+j}},\omega_2(u_2)^{q^{i(z)+i}}\right]\\
&=&\prod_{j=0}^{d_1-1} \prod_{i=0}^{d_2-1}
\left[\omega_1(u_1)^{q^{i+j}},\omega_2(u_2)^{q^i}\right]
\end{eqnarray*}
where $i(x),j(x),i(z),j(z)$ are all taken in the $m_1$-by-$m_2$
rectangle. Thus it suffices to show that the last product above is
equal to the number of solutions to a single equation
$U_1X_{rs}=X_{rs}U_2$.

  Let $f_i$ be the minimal polynomial over $\FF_q$ of $u_i$ for
  $i=1,2$.  Since $f_i$ is irreducible over $\FF_q$, the invariant
  factors of $\lambda I_{d_i}-U_i$ are $1,\dots,1,f_i$ for $i=1,2$.
  The roots of $f_i$ are $u_i,u_i^q,\ldots,u_i^{q^{d_i-1}}$,
  which are distinct and form the orbit of $u_i$ under the Frobenius
  automorphism $u \mapsto u^q$.

  Suppose that $f_1\neq f_2$. Then $u_1^{q^i} \neq u_2^{q^j}$
  for all $i,j$. Hence
  \begin{eqnarray*}
  \prod_{j=0}^{d_1-1} \prod_{i=0}^{d_2-1}
  \left[\omega_1(u_1)^{q^{i+j}},\omega_2(u_2)^{q^{i}}\right]
  &=& \prod_{j=0}^{d_1-1} \prod_{i=0}^{d_2-1}
  \frac{\omega_1(u_1)^{q^{i+j+1}}-\omega_2(u_2)^{q^{i+1}}}
  {\omega_1(u_1)^{q^{i+j}}-\omega_2(u_2)^{q^{i}}}\\
  &=& \frac{ \prod_{j=0}^{d_1-1} \prod_{i=1}^{d_2}
  \left(\omega_1(u_1)^{q^{i+j}}-\omega_2(u_2)^{q^{i}}\right)}
  { \prod_{j=0}^{d_1-1} \prod_{i=0}^{d_2-1}
  \left(\omega_1(u_1)^{q^{i+j}}-\omega_2(u_2)^{q^{i}}\right)}\\
  &=&1.
  \end{eqnarray*}
  Here again the last equality follows from
  $\omega_i(u_i)^{d_i}=\omega_i(u_i)$. Thus, our product formula
  predicts that the number of solutions $N$ is $1$, which agrees with
  Cecioni--Frobenius. 

  Now assume $f_1=f_2=f$, which implies $d_1=d_2=d$. Since $u_1$ and
  $u_2$ are both roots of $f$, there exists a unique $k$ such that
  $0\leq k\leq d-1$ and $u_1^{q^k}=u_2$. Hence, if $0\leq j\leq d-1$ then
  \[
  \left[\omega_1(u_1)^{q^j},\omega(u_2)\right]_{q^d}=
  \left\{\begin{array}{ll}
      1, & j\ne k,\\
      q^d, & j=k.
    \end{array}\right.
  \]
  One also sees from the definition that
  \[
  \prod_{i=0}^{d-1}\left[t_1^{q^{i+j}},t_2^{q^{i}}\right]_q
  =\left[t_1^{q^j},t_2\right]_{q^d}.
  \]
  We predict that,
  \[
  N=\prod_{i,j=0}^{d-1}\left[\omega(u_1)^{q^{i+j}}, \omega(u_2)^{q^{i}}\right]_q
    =\prod_{j=0}^{d-1}\left[\omega_1(u_1)^{q^j},\omega(u_2)
    \right]_{q^d} =q^d.
  \]
  This agrees with Cecioni--Frobenius, which gives the dimension
  of the solution space of $U_1 X_{rs}= X_{rs} U_2$ to be $d$.
\end{proof}

We are finally in a good position to prove our theorem. Recall that
we have fixed injections $\omega_r : \FF_{q^{\alpha_r}}^\times
\hookrightarrow \CC^\times$ of groups throughout this discussion.
\begin{proof}[Proof of Theorem~\ref{thm:main}]
  Let $\beta=(\beta_1,\ldots,\beta_\ell)$ be a weak composition of $k$
  with $\beta_r\leq \alpha_r$ for $r=1,\ldots,\ell$. It suffices to
  show that the number of subspaces $Z$ that are fixed by an element
  $u=(u_1,\ldots,u_\ell)$ in $T_\alpha$ and have $\beta(Z)=\beta$
  is equal to
  \[
  \sum_{\lambda:\beta(\lambda)=\beta} \wt(\lambda;\alpha,k)|_{t_r =
    \omega_r(u_r)}.
  \]

  Let $Z$ be such a subspace, and let $[Z_{rs}]_{1 \leq r \leq s \leq \ell}$
  be the row echelon form of $Z$ given by Lemma~\ref{lem:structure}. Using
  Corollary~\ref{cor:action} one has
  \[
  Zu=[Z_{rs}]\cdot \operatorname{diag}(U_1,\dots,U_1,U_2,\dots,U_2,\dots ,
  U_\ell,\dots U_\ell)
  \]
  where each $U_r$ appears $\alpha_r/d_r$ times.
  Since the pivots in $[Z_{rs}]$ form identity blocks $I_{d_r}$, one
  can bring $Zu$ back to a row echelon form by left-multiplying by
  \[
  \operatorname{diag}(U_1,\ldots,U_1,\dots,U_2,\dots,U_2,\ldots,U_\ell,\ldots,U_\ell)^{-1},
  \]
  where each $U_r$ appears $\beta_r/d_r$ times, and the result must
  agree with the row-echelon form of $Z$. Using
  Lemma~\ref{UpperBlock} one sees that the number of
  solutions for $Z_{rs}$, $r<s$, is given by
  \begin{equation}\label{eq:upper}
  \wt(\lambda^{r,s};\omega_r(u_r),\omega_s(u_s))
  =\prod_{j=0}^{\beta_r-1}\prod_{i=0}^{\alpha_s-\beta_s-1}
    [\omega_r(u_r)^{q^{i+j}},\omega_s(u_s)^{q^{i+\beta_r}}].
  \end{equation}
  where $\lambda^{r,s}$ is a $\beta_r$-by-$(\alpha_s-\beta_s)$
  rectangle.

  The number of solutions for the anti-diagonal block $Z_{rr}$
  to be fixed by $u_r$ is the number of $\beta_r$-dimensional
  $\FF_q$-subspaces of  $\FF_{q^{\alpha_r}}$ fixed by $u_r$,
  which was computed by Reiner and Stanton to be an evaluation
  of their ($q,t$)-binomial coefficient. They proved in their
  \cite[Theorem~5.2]{ReinerStanton} that this number is
  \begin{equation}
    \label{eq:diagonal}
    \sum_{\lambda^{r,r} \subset (\alpha_r - \beta_r)^{\beta_r}}
    \wt(\lambda^{r,r};\omega_r(u_r),\omega_r(u_r)).
  \end{equation}

  Combining equations~\eqref{eq:upper} and \eqref{eq:diagonal}
  together with the bijection (\ref{eq:lambda}) shows
  that the total number of subspaces $Z$ with $\beta(Z)=\beta$ that
  are left fixed by $u$ is
  \[
  \sum_{\lambda: \beta(\lambda)=\beta} \wt(\lambda;\alpha,k)|_{t_r =
    \omega_r(u_r)}.
  \]
  This is what we needed to show.
\end{proof}

\begin{corollary}\label{cor:main}
Let $C_\lambda$ be the Schubert cell of $Gr_k(V_\alpha)$ indexed by
$\lambda\vdash n$. For any weak composition
$\beta=(\beta_1,\ldots,\beta_\ell)$ of $k$ with
$\beta_r\leq\alpha_r$, $1\leq r\leq \ell$, the triple
\[
\left( \bigcup_{\lambda:\beta(\lambda)=\beta} C_\lambda,\
\sum_{\lambda:\beta(\lambda)=\beta} \wt(\lambda;\alpha,k),\ T_\alpha
\right)
\]
exhibits the cyclic sieving phenomenon. In addition, the polynomial
$\sum_{\lambda:\beta(\lambda)=\beta} \wt(\lambda;\alpha,k)$ factors as
\[
 \left(\prod_{r=1}^\ell \prod_{i=0}^{\beta_r-1}
   \frac{t_r^{q^{\alpha_r}} - t_r^{q^i}}{t_r^{q^{\beta_r}} - t_r^{q^i}}\right)
 \left(\prod_{1 \leq r< s \leq \ell}\prod_{j=0}^{\beta_r-1}\prod_{i=0}^{\alpha_s-\beta_s-1}
    [t_r^{q^{i+j}},t_s^{q^{i+\beta_r}}]\right).
\]
\end{corollary}

\begin{proof}
One can refine the action of $T_\alpha$ to
$\bigcup_{\beta(\lambda)=\beta} C_\lambda$ since
\[
\pi_r(\ker(\pi_{r-1})\cap Zu)=\pi_r(\ker(\pi_{r-1})\cap Z)u.
\]
The proof of Theorem~\ref{thm:main} shows precisely the cyclic
sieving phenomenon for this refined action. The second assertion
follows from the product formulation of the ($q,t$)-binomial
coefficient (Reiner--Stanton \cite[p.1]{ReinerStanton}) and
Equation~\eqref{eq:upper}.
\end{proof}

\begin{remark}
It follows the above corollary that
\begin{eqnarray*}
\qbinom{n}{k}{q}&=&
\lim_{t_1,\ldots,t_\ell\to1}\sum_{\beta_1+\cdots\beta_\ell=k}
\sum_{\lambda:\beta(\lambda)=\beta}\wt(\lambda;\alpha,k)\\
&=& \sum_{\beta_1+\cdots\beta_\ell=k} \prod_{r=1}^\ell
\qbinom{\alpha_r}{\beta_r}{q}\,
 \prod_{1 \leq r< s \leq \ell} q^{\beta_r(\alpha_s-\beta_s)}.
\end{eqnarray*}
When $\ell=2$ this gives the well-known $q$-Vandermonde identity.
\end{remark}

\section{Partial Flag Varieties}\label{sec:flags}
In this section we generalize the previous results to the partial
flag varieties. We do this at the cost of some repetition, as all of
our previous results are subsumed in the forthcoming pages. We find
this to be pedagogically sound since the proofs presented by
themselves would be opaque without the presentation of the
Grassmannians as ``warm-up'' cases.

We start by giving the relevant definitions, and then consider the
two extreme cases of the $(1^n)$ and $(n)$ tori, $(\FF_q^\times)^n$
and $\FF_{q^n}^\times$. Following this we define the polynomials
which gives the CSP for the partial flag varieties and prove our
main theorem.

\subsection{Partial flag varieties and Schubert decomposition}
Our discussion begins with some geometry and combinatorics of
partial flag varieties.

Let $\beta=(\beta_1,\ldots,\beta_m)$ be a (weak) composition of
$n$. The {\it partial flag variety} of type $\beta$ is
\[
\Fl(\beta)=\left\{ 0\subset V_1\subset V_2\subset\cdots\subset V_m=V:
  \dim(V_i)=\beta_1+\cdots+\beta_i \right\},
\]
where $V$ is a $n$-dimensional vector space over $\FF_q$. We will
usually take $V=V_\alpha$, as in Section~\ref{sec:proof}.

The \emph{parabolic subgroup} (or \textit{Young subgroup})
$W_\beta=\SS_\beta$ of the symmetric group $W=\SS_n$ is the direct
product $\mathfrak S_{\beta_1}\times\cdots\times\mathfrak
S_{\beta_m}$ where $\mathfrak S_{\beta_i}$ is the permutation group
on the integers
$$\beta_1+\cdots+\beta_{i-1}+1,\ldots,\beta_1+\cdots+\beta_i.$$
Written in one-line notation, the permutations $w=w(1)w(2)\ldots
w(n)$ in $\SS_n$ are naturally partitioned by $\beta$ into blocks.
Any coset $w\SS_\beta$ can be represented by the element with the
minimal length, {\it i.e.}, the element obtained from $w$ by sorting
every block of $w$ into increasing order. For example, if
$\beta=(4,2,2)$ and $w=5268|73|14$ then the minimal coset
representative of $w\SS_\beta$ is $2568|37|14$. Let
$W^\beta=W/W_\beta$ be the set of all these minimal coset
representatives.

Let $0\subset V_1\subset V_2\subset\cdots\subset V_m$ be a flag in
$\Fl(\beta)$ represented by an $n$-by-$n$ matrix $F$ whose first
$\beta_1+\cdots+\beta_i$ rows span $V_i$ for all $i=1,\ldots,m$.
Then there exists a unique permutation $w\in W^\beta$ such that $F$
can be reduced by row operations fixing the partial flag to a
\emph{row echelon form} $[a_{ij}]_{i,j=1}^n$ with
\begin{itemize}
\item $a_{ij}=1$, called a \emph{pivot}, if $j=w(i)$,
\item $a_{ij}$ is arbitrary if $i<w^{-1}(j)$, $w(i)>j$,
\item $a_{ij}=0$ otherwise.
\end{itemize}
In other words, $F$ is obtained from the permutation matrix of $w$
by replacing those zeros in the positions corresponding to the
inversions of $w$ with arbitrary numbers in $\FF_q$.
The \textit{Schubert cell} $C_w$ indexed by $w \in W^\beta$ consists
of those flags whose associated permutation is $w$.

For instance, if $\beta=(1,2,2)$ and $w=23514$, then $C_w$ consists
of all partial flags that can be represented in the form
\[
\left[\begin{tabular}{ccccc}
$*$ & $1$ & $0$ & $0$ & $0$ \\
\hdashline
$*$ & $0$ & $1$ & $0$ & $0$ \\
$*$ & $0$ & $0$ & $*$  & $1$ \\
\hdashline
$1$ & $0$ & $0$ & $0$ & $0$ \\
$0$ & $0$ & $0$ & $1$ & $0$
\end{tabular}\right].
\]
Here the dashed lines indicate which rows span subspaces of the associated partial flag. We shall omit
them if they are clear from context. We refer to Fulton \cite[\S
10.2]{Fulton} for the geometry of the complete flag variety.
\subsection{The $1^n$-torus action on $\Fl(\beta)$}
The torus $T_{1^n}$ acts on $\Fl(\beta)$ by rescaling the columns of a
matrix representing a flag. For example, if $F$ is a flag as in the
above example, then
\[
F=\left[\begin{tabular}{ccccc}
$a$ & $1$ & $0$ & $0$ & $0$ \\
$b$ & $0$ & $1$ & $0$ & $0$ \\
$c$ & $0$ & $0$ & $d$  & $1$ \\
$1$ & $0$ & $0$ & $0$ & $0$ \\
$0$ & $0$ & $0$ & $1$ & $0$
\end{tabular}\right]
\xrightarrow{u} F\cdot u = \left[\begin{tabular}{ccccc}
$au_1$ & $u_2$ & $0$ & $0$ & $0$ \\
$bu_1$ & $0$ & $u_3$ & $0$ & $0$ \\
$cu_1$ & $0$ & $0$ & $du_4$  & $u_5$ \\
$u_1$ & $0$ & $0$ & $0$ & $0$ \\
$0$ & $0$ & $0$ & $u_4$ & $0$
\end{tabular}\right].
\]
Applying row operations that do not effect $F\cdot u$ we obtain
\[
\left[\begin{tabular}{ccccc}
$u_2^{-1}au_1$ & $1$ & $0$ & $0$ & $0$ \\
$u_3^{-1}bu_1$ & $0$ & $1$ & $0$ & $0$ \\
$u_5^{-1}cu_1$ & $0$ & $0$ & $u_5^{-1}du_4$  & $1$ \\
$1$ & $0$ & $0$ & $0$ & $0$ \\
$0$ & $0$ & $0$ & $1$ & $0$
\end{tabular}\right].
\]
By the uniqueness of the row echelon form, $F$ is fixed by $u$ if
and only if
\[
u_2a=au_1,\ u_3b=bu_1,\ u_5c=cu_1,\ u_5d=du_4.
\]
The number of solutions to these equations is given by
\[
[t_2,t_1][t_3,t_1][t_5,t_1][t_5,t_4]|_{t_i=\omega(u_i)}
\]
where $\omega:\FF_q^\times\hookrightarrow \CC^\times$ is a fixed
injection of groups.

One can easily extend this example to the action of $1^n$-torus on the
partial flag variety $\Fl(\beta)$ for all compositions $\beta$ of $n$
and show that
\[
\left( \Fl(\beta),\ X_{1^n,\beta}(t),\
(\mathbb F_q^\times)^{\times n} \right)
\]
exhibits the cyclic sieving phenomenon, where
$$
X_{1^n,\beta}(t)=\sum_{w\in W^\beta} \prod_{(i,j)\in{\rm Inv}(w)}
[t_i,t_j].
$$
Here ${\rm Inv}(w)=\{(i,j):i<j,w(i)>w(j)\}$ is the usual set of
inversions of $w$.

\subsection{The $n$-torus action on $\Fl(\beta)$}\label{sec:n torus on flags}
Reiner, Stanton, and White \cite{RSW} considered the cyclic action
of $\mathbb F_{q^n}^\times$ on the partial flag variety
$\Fl(\beta)$. They observed that a flag $0\subset
V_1\subset\cdots\subset V_m=V$ in $F(\beta,\FF_q)$ is fixed by an
element $u$ in $\FF_{q^n}^\times$ if and only if all $V_i$ are
$\FF_{q^d}$-spaces, where $d=[\FF_q[u]:\FF_q]$. On the other hand,
they defined the \emph{$(q,t)$-multinomial coefficient}
\[
\qbinom{n}{\beta}{q,t}:=\frac{n!_{q,t}}
{\beta_1!_{q,t}\cdot\beta_2!_{q,t^q}\cdot\beta_3!_{q,t^{q^2}}\cdots},
\]
where
$n!_{q,t}=(1-t^{q^n-1})(1-t^{q^n-q})\cdots(1-t^{q^n-q^{n-1}})$, and
showed that
\[
\qbinom{n}{\beta}{q,\omega(u)}=\qbinom{n/d}{\beta/d}{q^d},
\]
where $\beta/d=(\beta_1/d,\beta_2/d,\ldots)$ and $\omega:
\FF_{q^n}^\times\hookrightarrow \CC^\times$ is an injection of
groups. Hence
\[
\left( \Fl(\beta),\ X_{n,\beta}(t),\ \mathbb F_{q^n}^\times \right)
\]
exhibits the cyclic sieving phenomenon, with
$X_{n,\beta}(t)=\qbinom{n}{\beta}{q,t}$.

Reiner and Stanton \cite[Section 8]{ReinerStanton} defined a weight
function that allows one to write the $(q,t)$-multinomial
coefficient as
\begin{equation}\label{qt}
\qbinom{n}{\beta}{q,t}= \sum_{w\in W^\beta} \wt(w;t).
\end{equation}
They used a recurrence relation to define the weight, which was later
shown by Hivert and Reiner \cite{HivertReiner} to take the form
\[
\wt(w;t)=\prod_{(i,j)\in {\rm Inv}(w)} \wt((i,j);t),
\]
for some weights associated to the inversions of $w$, which we now
define.

Given a word $w=w_1\ldots w_\ell$, recursively define a labeled tree
by taking the smallest letter $w_s$ of $w$ as the root and attaching
to it the trees obtained from the subwords $w_1\ldots w_{s-1}$ and
$w_{s+1}\ldots w_\ell$ as left and right subtrees. For instance, the
tree associated to $w=385216479$ is
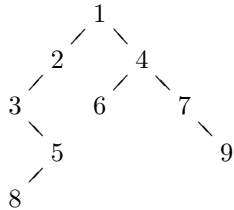
\[
\xymatrix @R=5pt @C=5pt {
   &    & 1 \ar@{-}[dl] \ar@{-}[dr] \\
   & 2 \ar@{-} [dl] & & 4 \ar @{-}[dl] \ar @{-}[dr] \\
3 &    & 6   & & 7 \ar @{-}[dr]\\
   & 5 \ar @{-}[ul] & & & & 9 \\
8 \ar @{-}[ur] }
\]
For any inversion $(i,j)$ of $w$, find the smallest $w(k)$ with
$i\leq k\leq j$. In the tree of $w$, $w(k)$ is the join (lowest
common parent) of $w(i)$ and $w(j)$. Let $\ell$ (resp. $r$) be the
set of all vertices in the left (resp. right) subtree of $w(k)$
whose label is at least $w(i)$ (resp. at most $w(j)$). Then
\[
{\rm wt}((i,j);t)=[k-1+r,k-1+r-\ell].
\]
Here the notation is $[a,b]=[t^{q^a},t^{q^b}]_q$.

The weights of the inversions of $w=385216479$ are given in the
following table.
\[
\begin{matrix}
(w(i),w(j)) &  k & \ell & r & {\rm wt} \\
\hline
(2,1) & 5 & 4 & 0 & [4,0] \\
(3,1) & 5 & 3 & 0 & [4,1] \\
(5,1) & 5 & 2 & 0 & [4,2] \\
(8,1) & 5 & 1 & 0 & [4,3] \\
(3,2) & 4 & 3 & 0 & [3,0] \\
(5,2) & 4 & 2 & 0 & [3,1] \\
(8,2) & 4 & 1 & 0 & [3,2]
\end{matrix}
\qquad
\begin{matrix}
(w(i),w(j)) &  k & \ell & r & {\rm wt} \\ \hline
(5,4) & 5 & 2 & 1 & [5,3] \\
(6,4) & 7 & 1 & 0 & [6,5] \\
(8,4) & 5 & 1 & 1 & [5,4] \\
(8,5) & 3 & 1 & 0 & [2,1] \\
(8,6) & 5 & 1 & 2 & [6,5] \\
(8,7) & 5 & 1 & 3 & [7,6] \\
\null
\end{matrix}
\]

Reiner and Stanton~\cite{ReinerStanton} also observed that
\[
\lim_{t\to 1}\qbinom{n}{\beta}{q,t}=\qbinom{n}{\beta}{q},
\]
\begin{equation}\label{q->1}
\lim_{q\to 1}\qbinom{n}{\beta}{q,t^{\frac{1}{q-1}}}
=\qbinom{n}{\beta}{t}.
\end{equation}

\subsection{Statement of the main result}
Let $\alpha$ and $\beta$ be compositions of $n$. The action of the
torus $T_\alpha$ on $V=V_\alpha$ induces an action on the partial
flag variety $\Fl(\beta)$. Our goal in this subsection is to define
a multivariate polynomial $X_{\alpha,\beta}(\tt)$ so that the triple
\[
\left( \Fl(\beta),\ X_{\alpha,\beta}(t),\ T_\alpha \right)
\]
exhibits the cyclic sieving phenomenon.

The compositions $\alpha$ and $\beta$ give set partitions of $[n]$
as in
$$
A_r=\{\alpha_1+\cdots+\alpha_{r-1}+1,\ldots,\alpha_1+\cdots+\alpha_r\}
$$
for $r=1,\ldots,\ell$, and likewise define $B_1,\ldots,B_m$ from
$\beta$. These will allow us to break up a permutation $w \in W^\beta$
into ``sub-permutations'', as we did with partitions in the
Grassmannian case.

Let $C_w$ be a Schubert cell of $\Fl(\beta)$ represented by its
associated row-echelon form $F=[a_{ij}]_{i,j=1}^n$ of $0$'s, $1$'s
and $*$'s. Define $F_{rs}$ to be the submatrix of $F$ with column
indices in $A_s$ and row indices in $w^{-1}(A_r)$, for $1\leq
r,s\leq \ell$.

If $F$ represents the partial flag $0\subset V_1\subset
\cdots\subset V_m=V_\alpha$, and
\[
\pi_s:
V_\alpha=\FF_{q^{\alpha_1}}\oplus\cdots\oplus\FF_{q^{\alpha_\ell}}
\rightarrow
\FF_{q^{\alpha_s}}\oplus\cdots\oplus\FF_{q^{\alpha_\ell}}.
\]
is the projection map, $1\leq s\leq \ell$, then define weak
compositions $\beta^{(s)}(F)$ by
\[
\beta^{(s)}_k(F)=\dim \pi_s(\ker(\pi_{s+1})\cap V_k)), 1\leq k\leq
m.
\]
This is equivalent to
\[
\beta_k^{(s)}(F)=|\{ i : i \in B_k, w(i) \in A_s \} |.
\]
One sees that $\beta^{(1)}(F),\ldots,\beta^{(\ell)}(F)$ are weak
compositions of $\alpha_1,\ldots,\alpha_\ell$, respectively, and
their component-wise sum is $\beta$, {\it i.e.}
\[
\beta^{(1)}_k(F)+\cdots+\beta^{(\ell)}_k(F)=\beta_k, \quad
k=1,\ldots,m.
\]

Since the definition depends only on the Schubert cell $C_w$ that
contains $F$, we can write $\beta^{(s)}(w)=\beta^{(s)}(F)$.
The matrix $[\beta^{(s)}(w)_k]_{k,s}$ indexes the double coset
$\SS_\alpha w\SS_\beta$.

Each ``diagonal'' submatrix $F_{ss}$ determines a permutation $w_s$
in $\SS_{\alpha_s}/\SS_{\beta^{(s)}(w)}$. Conversely, given
permutations $w_s$ in $\SS_{\alpha_s}/\SS_{\beta^{(s)}(w)}$, $1\leq
s\leq \ell$, one can recover the permutation $w$ in $\SS_n$ in a
unique way.

\begin{example}
Let $\alpha=(4,4)$, $\beta=(1,3,4)$, and $w=53461278$. Then $C_w$ is
represented by
\[
F=\left[\begin{array}{cccc|cccc}
*&*&*&*&1& & & \\ \hdashline
*&*&1& & & & & \\
*&*& &1& & & & \\
*&*& & & &1& & \\ \hdashline
1& & & & & & & \\
 &1& & & & & & \\
 & & & & & &1& \\
 & & & & & & &1\\
  \end{array}\right].
  \]
It is divided by $\alpha$ into submatrices
\[
F_{11}=\left[\begin{array}{cccc}\hdashline
*&*&1&  \\
*&*& &1 \\
\hdashline
1& & & \\
 &1& &\\
  \end{array}\right],\quad
F_{12}=0,\quad
F_{21}=\left[\begin{array}{cccc}
*&*&*&* \\
\hdashline
*&*& &  \\
\hdashline
 & & & \\
 & & &\\
  \end{array}\right],\quad
F_{22}=\left[\begin{array}{cccc}
1& & & \\ \hdashline
 &1& & \\ \hdashline
 & &1& \\
 & & &1\\
  \end{array}\right].
\]
One sees that $\beta^{(1)}(w)=(0,2,2)$, $\beta^{(2)}(w)=(1,1,2)$,
$w_1=3412$, $w_2=5678$.
\end{example}

\begin{lemma}\label{BlockForm}
\noindent (a) If $r<s$ then $F_{rs}=0$.

\noindent (b) If $r=s$ then $F_{ss}$ is the row echelon form for the
Schubert cell $C_{w_s}$ of $\Fl(\beta^{(s)}(w))$.

\noindent (c) If $r>s$ then $F_{rs}$ contains only stars and zeros.

\noindent (d) Let $a_{ij}$ be a star in $F$ that falls into some
$F_{rs}$ with $r>s$. If $i\in B_k$, then $a_{i'j}$ is also a star
for all $i'\in B_k$ with $w(i')\in A_r$; if $w^{-1}(j)\in B_k$ then
$a_{ij'}$ is a star for all $j'\in A_s$ with $w^{-1}(j')\in B_k$.
\end{lemma}

\begin{proof}
  By the definition, $a_{i,w(i)}$ is a $1$ for $i=1,\ldots,n$,
  $a_{ij}$ is a star, whenever $i<w^{-1}(j)$ and $w(i)>j$, and
  $a_{ij}$ is a $0$ otherwise. This at once yields (a) and (c) and a
  moment's thought gives (b).


Finally, to prove (d), let $a_{ij}$ be a star that falls in some
$F_{rs}$ with $r>s$, {\it i.e.} $i<w^{-1}(j)$, $w(i)>j$, $w(i)\in
A_r$, $j\in A_s$.

Suppose that $i$ and $i'$ are both in $B_k$ for some $k$, and
$w(i')$ is in $A_r$. If $w^{-1}(j)\in B_k$, then $i\in B_k$,
$i<w^{-1}(j)$, and $w(i)>j$ give a contradiction to $w\in W^d$.
Hence $w^{-1}(j)\notin B_k$, and then $i<w^{-1}(j)$ implies
$i'<w^{-1}(j)$. Since $w(i),w(i')\in A_r$, $j\in A_s$, and $w(i)>j$,
one also has $w(i')>j$. Therefore $a_{i'j}$ is a star.

Similarly, if $w^{-1}(j)$ and $w^{-1}(j')$ are both in $B_k$ for
some $k$, and $j'$ is in $A_s$, then $i\notin B_k$ and $i<w^{-1}(j)$
imply $i<w^{-1}(j')$, and $w(i)>j$ implies $w(i)>j'$. Thus $a_{ij'}$
is a star.
\end{proof}

It follows from (d) that each $F_{rs}$, $r>s$, consists of one
$\beta^{(s)}_b$-by-$\beta^{(r)}_a$ rectangle of stars for all pairs
$(a,b)$ with $1\leq a<b\leq m$, and zeros in the remaining spots.
Define its weight to be
\[
\wt(F_{rs};t_r,t_s):=\prod_{1\leq a<b\leq m}
\prod_{i=1}^{\beta^{(r)}_a}\prod_{j=1}^{\beta^{(s)}_b}
\,\left[t_r^{i+j},t_s^{i+\beta^{(s)}_b}\right].
\]
In the previous example, $F_{21}$ contains three $1$-by-$2$
rectangles of stars, so its weight is
\[
\wt(F_{21};t_2,t_1)=\left([t_2,t_1^q][t_2^q,t_1^{q^2}]\right)^3.
\]
Let $\wt(F_{ss};t_s,t_s):=\wt(w_s;t_s)$ as in (\ref{qt}), $1\leq
s\leq \ell$. Define the \emph{weight of a Schubert cell $C_w$} to be
\[
\wt(w;\alpha):=\prod_{1\leq s\leq r\leq \ell}\wt(F_{rs};t_r,t_s).
\]
This weight does not depend on $\beta$. Finally define
\[
X_{\alpha,\beta}(\tt):=\sum_{w\in W^\beta} \wt(w;\alpha).
\]

Now we can state the main result on flag varieties.

\begin{theorem}\label{thm:flags}
  The triple $(\Fl(\beta),\ X_{\alpha,\beta}(\tt),\ T_\alpha)$ exhibits the
    cyclic sieving phenomenon.
\end{theorem}

\subsection{Proof of the main result}

The reader can check that the following is a straightforward
generalization of Lemma~\ref{lem:structure} to flag varieties.

\begin{lemma}\label{lem:FlagStructure}
Let $u=(u_1,\ldots,u_\ell)$ be an element in $T_\alpha$ with
$[\FF_q[u_s]:\FF_q]=d_s$, $s=1,\ldots,\ell$, and let $F$ be a flag
in $\Fl(\beta,\FF_q)$ fixed by $u$. Then under the basis for
$V_\alpha$ given in Corollary~\ref{cor:action}, $F$ has row echelon
form $[F_{rs}]_{r,s=1}^\ell$ in which the pivots form block matrices
equal to $I_{d_s}$.
\end{lemma}

We are now in a good position to prove the main result on the flag
varieties.

\begin{proof}[Proof of Theorem~\ref{thm:flags}]
  Take weak compositions $\beta^{(1)},\dots,\beta^{(\ell)}$ of
  $\alpha_1,\dots,\alpha_\ell$, respectively, such that the
  component-wise sum of $\beta^{(1)},\ldots,\beta^{(\ell)}$ is equal
  to $\beta$. Let $u=(u_1,\ldots,u_\ell)$ be an element in $T_\alpha$
  with $[\FF_q[u_s]:\FF_q]=d_s$ for $s=1,\ldots,\ell$. Fix the
  embeddings $\omega_s: \FF_q^{\alpha_s}\hookrightarrow
  \CC^\times$. Consider all flags $F$ in $\Fl(\beta)$ which are fixed
  by $u$ and have $\beta^{(s)}(F)=\beta^{(s)}$, $s=1,\ldots,\ell$.


  Let $[F_{rs}]_{r,s=1}^\ell$ be the row echelon form of $F$ given
  by Lemma~\ref{lem:FlagStructure}. Similarly to the proof of
  Theorem~\ref{thm:main}, it follows from Corollary~\ref{cor:action},
  Lemma~\ref{BlockForm}(a), and Lemma~\ref{lem:FlagStructure} that
  $F_{ss}$ is fixed by $u_s$ for $s=1,\dots,\ell$, and
  \begin{equation}\label{eq:F_rs}
    {\rm diag}(U_r,\dots,U_r)F_{rs}=
    F_{rs}{\rm diag}(U_s,\dots,U_s)
  \end{equation}
  for all $1\leq s<r\leq \ell$.  By equation (\ref{qt}), the
  number of choices for $F_{ss}$ is
  $$
  \qbinom{\alpha_s}{\beta^{(s)}}{q,\omega_s(u_s)}= \sum_{w_s\in
    \SS_{\alpha_s}/\SS_{\beta^{(s)}}} \wt(w_s;\omega_s(u_s)).
  $$
By Lemma~\ref{BlockForm} (d) and Lemma~\ref{UpperBlock}, the number
of solutions for $F_{rs}$ is
\[\wt(F_{rs};\omega_r(u_r),\omega_s(u_s)).\]

Multiplying all choices for the various submatrices $F_{rs}$ one
obtains the number of solutions for $F$ from
\[
\sum_{w:(\forall s)\beta^{(s)}(w)=\beta^{(s)}} \wt(w;\alpha)
\]
by setting $t_s=\omega_s(u_s)$, $1\leq s\leq \ell$.
\end{proof}

\begin{remark}\label{rem:CSPflag}
  Similarly to Corollary~\ref{cor:main}, and as the above proof shows,
  for fixed weak compositions $\beta^{(1)},\ldots,\beta^{(\ell)}$ of
  $\alpha_1,\ldots,\alpha_\ell$ whose component-wise sum equals
  $\beta$, one can refine the cyclic sieving phenomenon to the triple
  \[
  \left(\bigcup_{w: (\forall r)\beta^{(r)}(w)=\beta^{(r)}}C_w,\
    \sum_{w: (\forall r)\beta^{(r)}(w)=\beta^{(r)}} \wt(\alpha;w),\
    T_\alpha\right)
  \]
  and the polynomial $\sum_{w: (\forall r)\beta^{(r)}(w)=\beta^{(r)}} \wt(w,\alpha)$
  factors nicely. Taking $t_1,\ldots,t_\ell\to 1$ leads to
  \[
  \qbinom{\alpha_1+\dots+\alpha_\ell}{\beta_1,\dots,\beta_m}{q}
  =\sum_{\beta^{(1)},\dots,\beta^{(\ell)}}
  \prod_{r=1}^\ell\qbinom{\alpha_r}{\beta^{(r)}}{q} \prod_{1\leq s\leq
    r\leq\ell} \prod_{1\leq i<j\leq m}q^{\beta^{(r)}_i\beta^{(s)}_j},
  \]
  summed over all weak compositions $\beta^{(r)}$ of $\alpha_r$,
  $1\leq r\leq\ell$, with component-wise sum
  $\beta^{(1)}+\dots+\beta^{(\ell)}=\beta$.  This generalizes the
  $q$-Vandermonde identity.

  On the other hand, by Equation (\ref{q->1}) and the
  following limit
  \[
  \lim_{q\to 1}\ [a,b]_q=1,\quad a\ne b,
  \]
  one can view $Y_{\alpha,\beta}(\tt)$, defined in Equation (\ref{eq:q=1}),
  as a ``$q=1$'' version of $X_{\alpha,\beta}(\tt)$.
\end{remark}

\section{Further Questions}

The partial flag variety $\Fl(\beta)$ can be identified with the
parabolic cosets $G/P_\beta$, where $G=GL(n,\FF_q)$ and $P_\beta$ is
the parabolic subgroup of all block upper triangular matrices with
invertible diagonal blocks of size $\beta_1,\dots,\beta_m$. If $C$
denotes $\FF_{q^n}^\times$ then Springer's theorem~\cite{Springer}
asserts that, as $\FF_q[G\times C]$-bimodules, the coinvariant
algebra $\FF_q[{\bf
  x}]/(\FF_q[{\bf x}]^G_+)$ and the group algebra $\FF_q[G]$ have the
same composition factors. The cyclic sieving phenomenon for the triple
$(G/P_\beta,\ X_{n,\beta}({\bf t}),\ C)$ is a consequence of
Springer's theorem, with
\[
X_{n,\beta}({\bf t})=\qbinom{n}{\beta}{q,t}= {\rm
Hilb}\left(\FF_q[{\bf x}]/(\FF_q[{\bf x}]^G_+),t\right).
\]
See Broer, Reiner, Smith, and Webb \cite{BRSW} for details and
generalizations.

Is there a Springer-type result for the $\FF_q[G\times
T_\alpha]$-module $\FF_q[G]$ which would imply the cyclic sieving
phenomenon for the $T_\alpha$-action on $G/P_\beta$? There might be
some clue suggested by the factorization of the sieving polynomial
in Corollary~\ref{cor:main} or Remark~\ref{rem:CSPflag} (even for
the ``$q=1$'' version).

Now let $G$ be a linear algebraic group defined over the algebraic
closure of $\FF_q$ and fix a maximal split torus. Let $J$ be a subset
of the positive roots of $G$. This subset $J$ defines a parabolic
subgroup $P_J \subset G$ and the quotient $G/P_J$ is a generalized
flag variety. The set of $\FF_q$-rational points of $G/P_J$ is
invariant under the natural action of a maximal torus $T \subset
G(\FF_q)$.  Finally, let $W$ denote the Weyl group of $G$, $W_J$ a
parabolic subgroup and $W^J$ the shortest length coset representatives
of $W/W_J$.

When $G=\GL_k$, we have succeeded in expressing all of our cyclic
sieving polynomials in the form
\[
\sum_{w \in W^J} \wt(w;T),
\] {\it i.e.}, to each coset representative $w$ we have associated a
polynomial weight depending on $T$, which is a product over those
positive roots made negative by $w$.  Can this be extended to groups
$G$ other than $\GL_n$?

\section{Acknowledgements}
The authors would like to thank Julian Gold for preliminary
computations done during a summer 2010 REU at UC Davis with the
first author. Thanks also to Victor Reiner and Dennis Stanton for
helpful comments and suggestions.

\end{document}